\documentclass[12pt, reqno]{amsart}
\usepackage{graphicx}
\usepackage{subfigure}
\usepackage{latexsym}
\usepackage{amsmath}
\usepackage{amssymb}
\usepackage{mathrsfs}
\usepackage{color}

 \newtheorem{theorem}{Theorem}[section]
 \newtheorem{Def}[theorem]{Definition}
 \newtheorem{Prop}[theorem]{Proposition}
 \newtheorem{Lem}[theorem]{Lemma}
 \newtheorem{Cor}[theorem]{Corollary}

 \numberwithin{equation}{section}
 
\begin{document}

\title{Moran sets and hyperbolic boundaries}

\author{Jun Jason Luo}
\address{Department of Mathematics,  Shantou
University, Shantou 515063, China} \email{luojun2011@yahoo.com.cn}

\thanks{The research is supported by STU Scientific Research Foundation for Talents (no. NTF12016).}

\keywords{Moran set, augmented tree,  hyperbolic boundary, Lipschitz
equivalence.}

\subjclass[2010]{Primary 28A80}
%
\date{\today}

\begin{abstract}
In the paper, we prove that a Moran set is homeomorphic to the
hyperbolic boundary of the representing symbolic space in the sense
of Gromov, which generalizes the results of Lau and Wang [{\it
Self-similar sets as hyperbolic boundaries}, Indiana U. Math. J.
{\bf 58} (2009), 1777-1795]. Moreover, by making use of this, we
establish the Lipschitz equivalence of a class of Moran sets.
\end{abstract}

\maketitle

\bigskip

\begin{section}{\bf Introduction}

For an iterated function system (IFS) of contractive similitudes on
${\Bbb R}^d$ and the associated self-similar set $K$, there is a
symbolic space which contains a natural tree structure and
represents every point of $K$. Kaimanovich \cite{Ka03} first
proposed  a {\em hyperbolic graph structure (called augmented tree)}
on the symbolic space of the Sierpinski gasket by adding horizontal
edges (corresponding to the (vertical) edges of the tree), and he
observed the relationship between the hyperbolic boundary and the
gasket. Lau and Wang (\cite{LaWa09}, \cite{Wa12}) developed this
idea to more general self-similar sets. It was proved that if an IFS
satisfies the open set condition (OSC), even weak separation
condition (WSC), then the augmented tree is hyperbolic in the sense
of Gromov and the hyperbolic boundary of the tree is shown to be
homeomorphic to $K$; moreover under certain mild condition, the
homeomorphism is actually a H\"older equivalent map. Recently, this
setup has been frequently used to study the random walks on trees
and their Martin boundaries  (\cite{Ki10}, \cite{JuLaWa12},
\cite{LaWa11}) and Lipschitz equivalence problem \cite{LaLu12}.

\medskip

As generalizations of  self-similar sets, Moran sets were introduced
by Moran \cite{Mo46}, which have abundant exotic fractal structures.
Let $\{n_k\}_{k\geq 1}$ be a sequence of positive integers and
$\{r_k\}_{k\geq 1}$ be a sequence of positive numbers satisfying
$n_k\geq 2, 0<r_k<1$ and $n_kr_k\leq 1$.  For any $k\geq 1$, let
$D_k=\Pi_{j=1}^k\{1,2,\dots,n_j\}=\{i_1\cdots i_k: 1\leq i_j\leq n_j, 1\leq j\leq k\}$ be the set
of words (or $k$-multi-indexes) and $D=\bigcup_{k\geq 0}D_k$ be the
set of all finite words (convention $D_0=\emptyset$), and let
$D_{\infty}=\{i_1 i_2\cdots : 1\leq i_j\leq n_j, j=1,2,\dots\}$ be
the set of all infinite words.  For integers $\ell>k\geq 1$, if ${\mathbf i}=i_1\cdots i_k\in D_k$
and ${\mathbf j}=j_1\cdots j_m\in \Pi_{j={k+1}}^\ell\{1,2,\dots,n_j\}$, we denote by $\mathbf
{ij}=i_1\cdots i_kj_1\cdots j_m\in D_{\ell}$ the concatenation.

\begin{Def}\label{def1.1}
Suppose that $J\subset {\mathbb R}^d$ is a compact set with nonempty
interior. The collection of subsets ${\mathcal F}=\{J_{\mathbf i}:
{\mathbf i}\in D\}$ of $J$ has the Moran structure, if it satisfies:

(i) $J_{\emptyset}=J$;

(ii) for any ${\mathbf i}\in D$, $J_{\mathbf i}$ is geometrically
similar to $J$, that is, there exists a similarity $S_{\mathbf i}:
{\mathbb R}^d \to {\mathbb R}^d$ such that $J_{\mathbf i}=S_{\mathbf
i}(J)$;

(iii) for any $k\geq 1$ and ${\mathbf i}\in D_{k-1}$,  $J_{{\mathbf
i}1},\dots, J_{{\mathbf i}n_k}$ are subsets of $J_{\mathbf i}$ and
$\text{int}(J_{{\mathbf i}i})\cap \text{int}(J_{{\mathbf
i}j})=\emptyset$ for $i\ne j$ where $\text{int}(A)$ denotes the
interior of a set $A$;

(iv) for any $k\geq 1$ and ${\mathbf i}\in D_{k-1}, 1\leq j\leq
n_k$, we have $$\frac{|J_{{\mathbf i}j}|}{|J_{\mathbf i}|}=r_k$$
where $|A|$ denotes the diameter of $A$.

\end{Def}

We call $E:= \bigcap_{k\geq 0} \bigcup_{{\mathbf i}\in
D_k}J_{\mathbf i}$  the \emph{(homogeneous) Moran set}.  For
${\mathbf i}\in D_{k-1}$ with $k\geq 1$, we call $J_{\mathbf i}$ a
\emph{basic set of order $k$} of the Moran set.  Let ${\mathcal
M}:={\mathcal M}(J, \{n_k\}, \{r_k\})$  denote the class of the
Moran sets satisfying (i)-(iv). From the definition above, if the
positions of the basic sets  are different, then the Moran sets are
different.  Compared with self-similar sets, Moran sets have more
fractal structures as following:

(1) the placements of the basic sets at each step of the geometric
construction can be arbitrary;

(2) the contraction ratios can be different at different steps;

(3) the cardinality of the basic sets in replacement at different
steps can be different.

\medskip

The systematical study on the geometric structure and dimension
theory  of Moran sets was developed by \cite{FeWeWu97},
\cite{HuRaWeWu00}, \cite{We01}.  It is well-known that all Moran
sets in ${\mathcal M}$ have the same Hausdorff and packing
dimensions provided $r :=\inf_k r_k>0$ \cite{FeWeWu97}. In our
consideration we always assume this condition holds.

\medskip

However, the relationship between Moran sets and the hyperbolic
structures of the representing symbolic spaces has not been
established yet. Similarly to the self-similar set, we can define
the corresponding symbolic space and  augmented tree of a Moran set
or class. Let
\begin{equation}\label{add.equ}
X_n=\{i_1\cdots i_k\in D: r_1\cdots r_k \leq r^n <
r_1\cdots r_{k-1}\}.
\end{equation}
Denote  $X=\bigcup_{n\geq 0}X_n$,
where $X_0=\emptyset$. Note that for $n\geq 0, n\ne \ell$,
$X_n\cap X_{\ell}=\emptyset$ holds. For each word ${\mathbf i}\in
X_n$, there exists a unique word ${\mathbf j}\in X_{n-1}$ and a
multi-index ${\mathbf k}=k_1\cdots k_{\ell}$ such that ${\mathbf
i}={\mathbf {jk}}$. Then we can define a vertical edge  in this way,
the set of all vertical edges is denoted by ${\mathcal E}_v$. The
horizontal edge is defined as follows: ${\mathbf i}\sim {\mathbf j}$
is called a {\em horizontal edge} if for some $ n\geq 1$ and
${\mathbf {i, j}}\in X_n$, $J_{\mathbf i}\cap J_{\mathbf j}\ne
\emptyset$, the set of all horizontal edges is denoted by ${\mathcal
E}_h$. Let ${\mathcal E}={\mathcal E}_v\cup {\mathcal E}_h$, then we
call $(X,{\mathcal E})$ an \emph{augmented tree}  (in Kaimanovich's
sense \cite{Ka03})  induced by the triplet $(J, \{n_k\}, \{r_k\})$
(or a Moran set $E\in {\mathcal M}$). There is a hyperbolic metric
$\rho_a$ (see Section 2) on $X$ which induces a {\em hyperbolic
boundary} $\partial X := \hat{X}\setminus X$ where $\hat{X}$ is the
completion of $X$ under the $\rho_a$.

\medskip

One of the main purposes of the present paper is to extend Lau and
Wang's results on self-similar sets \cite{LaWa09} to that on Moran
sets as following.

\begin{theorem}\label{th1.2}
Let $E$ be a Moran set, $(X,{\mathcal E})$ be the induced augmented
tree. Then

(i) $(X,{\mathcal E})$ is a hyperbolic graph in the sense of Gromov.

(ii) $E$ is homeomorphic to the hyperbolic boundary $\partial
(X,{\mathcal E})$. Furthermore, the H\"older equivalence holds if we
assume the additional condition (H) (see Section 3).
\end{theorem}

In \cite{LaLu12}, we investigated the Lipschitz equivalence  of
self-similar sets and self-affine sets by employing the structure of
the augmented tree and its hyperbolic boundary.  Based on the same
goal,  in this paper, we also establish the Lipschitz equivalence
relationship for a class of Moran sets.

\medskip

Let $J\subset {\mathbb R}^d$ be a connected compact set with $\text{int}(J)\ne \emptyset$
and let $r_k\equiv r\in (0,1)$ in Definition \ref{def1.1}. In this case, $X_n= D_n$ for $n\geq 0$.  We define
a special  Moran class ${\mathcal M}':={\mathcal M}(J, \{n_k\}, r)$
if in addition two more conditions hold:

(v) There exists $L\in {\mathbb N}$ such that for any $k\geq 1$ and a subset
$T \subset D_{k-1}$, if $\bigcup_{{\mathbf i\in T}}J_{\mathbf
i}$ is a connected component of $\bigcup_{{\mathbf i\in
D_{k-1}}}J_{\mathbf i}$ topologically,  then $\#T\leq L$.

(vi) For any $k\geq 1,\ T \subset D_{k-1}$ and ${\mathbf i}\in T$, if $J_{{\mathbf i}1},\dots, J_{{\mathbf i}n_k}$ are subsets of
$J_{\mathbf i}$ in the next step, then the union of all subsets of $\bigcup_{{\mathbf
i\in T}}J_{\mathbf i}$ can be written as $\#T (:=b)$ disjoint groups
as follows
$$\bigcup_{{\mathbf i}\in T}\bigcup_{j=1}^{n_k}J_{{\mathbf i}j}= \left({\bigcup}_{k\in \Lambda_1}C_{k}\right)\ \cup \ \cdots \cup \
\left({\bigcup}_{k\in \Lambda_{b}}C_{k}\right)$$   where $C_k$ are
connected components of $\bigcup_{{\mathbf i}\in
T}\bigcup_{j=1}^{n_k}J_{{\mathbf i}j}$ topologically such that every group
contains exactly $n_k$ terms of $J_{{\mathbf i}j}$.

\medskip

Since the Moran structure closely depends on the positions of the
basic sets in the constructing process, we can choose proper basic
sets in each step to satisfy conditions (v) and (vi). For example,
the fractal sets generated by a sequence of nested intervals or
squares (see \cite{Luo}, \cite{XX}).

\begin{theorem}\label{th1.3}
For any two Moran sets $E, F\in {\mathcal M}'$ defined as above. If
the condition in (ii) of Theorem \ref{th1.2} holds,  then $E$ and $F$ are
Lipschitz equivalent.
\end{theorem}

Actually we prove a less restrictive form of Theorem \ref{th1.3} (Theorem \ref{th4.6}) in terms of the rearrangeable augmented tree.

\medskip

The rest of the paper is organized as follows. In Section 2, we recall some
well-known results about hyperbolic graphs. In Section 3, we
identify Moran sets with hyperbolic boundaries and prove Theorem \ref{th1.2}.  We introduce a concept of `rearrangeable augmented tree' to
prove Theorems \ref{th1.3} in Section 4.

\end{section}

\bigskip

\begin{section}{\bf Hyperbolic graphs}

Let $X$ be a countably infinite set, we say that $X$ is a {\it
graph} if it is associated with a symmetric subset  ${\mathcal E}$
of $(X \times X) \setminus \{(x,x) : \ x \in X\}$, and call  $x \in
X$ a {\it vertex}, $(x, y) \in {\mathcal{E}}$ an {\it edge}, which
is more conveniently denoted by $x \sim y$ (intuitively, $x, y$ are
neighborhood to each other).  By a {\it path} in $X$ from $x$ to
$y$, we mean a finite sequence $x = x_0, x_1, \dots , x_n=y$ such
that $x \sim x_{i+1}, i =0, \dots, n-1$. We always assume that the
graph $X$ is connected, i.e., there is a path joining any two
vertices $x, y \in X$. We call $X$ a {\it tree} if the path between
any two points is unique. We equip a graph $X$ with an
integer-valued metric $d(x,y)$, which is the minimum among the
lengths of the paths from $x$ to $y$, and denote the corresponding
geodesic path by $\pi(x,y)$. We also use $|\pi(x,y)|$ to be the
length of the geodesic, which equals $d(x,y)$. Let $ o \in X$ be a
fixed point in $X$ and call it the {\it root} of the graph. We use
$|x|$ to denote $d(o,x),$ and say that $x$ belongs to $n$-th level
if $d(o,x)=n$.

\medskip

A graph $X$ is called {\it hyperbolic} (with respect to $o$) if
there is $\delta >0$ such that
$$
 |x \wedge y| \geq \min\{|x \wedge z|,  |z\wedge y|\}-\delta \qquad \forall \ \  x,y,z \in
 X,$$
where $ |x \wedge y| :=\frac{1}{2}(|x|+|y|-d(x,y))$ is the {\em
Gromov product} (\cite{Gr87}, \cite{Wo00}). For a fixed $a>0$ with
$a'=\exp(\delta a)-1<\sqrt{2}-1$, we define a {\it hyperbolic metric}
${\rho_a}(\cdot,\cdot)$ on $X$ by
\begin{equation}\label{eq2.1}
{\rho_a}(x,y)=\left\{
\begin{array}{ll}
\exp(-a |x \wedge y| ) & \quad \textrm{if \ $x\ne y$},\\
0  & \quad \textrm{otherwise}.
\end{array} \right.
\end{equation}

Since the hyperbolic metric $\rho_a$ is equivalent to a metric of $X$
with the same topology as long as $a'<\sqrt{2}-1$ \cite{Wo00}, we
always take $\rho_a$ as a metric for simplicity.  Under this metric
we then can complete the space $X$ and denote by the completion
$\hat{X}$. We call $\partial X := \hat{X}\setminus
 X$ the {\em hyperbolic boundary } of $X$. The metric $\rho_a$ can be extended onto $\partial X$, and under which
$\partial X$ is a compact set.  It is often useful to identify $\xi
\in \partial X$ with the {\it geodesic rays} in $X$ that converge to
$\xi$, i.e., an infinite path $\pi[x_1,x_2,\dots]$ such that
$x_i\sim x_{i+1}$ and any finite segment of the path is a geodesic.
It is known that two geodesic rays $\pi[x_1,x_2,\dots],
\pi[y_1,y_2,\dots]$ represent the same $\xi \in \partial X$ if and
only if $ |x_n\wedge y_n|\rightarrow\infty$ as $n\rightarrow\infty.$

\medskip

Our interest is on the following tree structure introduced by
Kamainovich which is used to study the self-similar sets (\cite
{Ka03}, {\cite {LaWa09}}).  For a tree $X$ with a root $o$, we  use
${\mathcal E}_v$ to denote the set of edges ($v$ for vertical). We introduce  additional edges on  each level $\{ x : \ d(0, x)
=n\}, \ n \in {\Bbb N}$ as follows. Let $x^{-k}$ denote the $k$-th ancestor of
$x$, the unique point in $(n-k)$-th level that is joined by a unique
path.
\medskip

\begin{Def}(\cite{Ka03}) \label{de2.2}
Let $X$ be a tree with a root $o$. Let \  ${\mathcal E}_h \subset (X
\times X) \setminus \{(x,x): \ x\in X\}$ such that it is symmetric
and satisfies :
$$
(x,y)\in {\mathcal E}_h \quad \Rightarrow \quad |x|=|y| \ \ \text{and}\ \
\text{either}~ x^{-1}=y^{-1} ~\text{or}~ (x^{-1},y^{-1})\in
{\mathcal E}_h.
$$
We call elements in  ${\mathcal E}_h$  horizontal edges,  and for
${\mathcal E}={\mathcal E}_v\cup{\mathcal E}_h$, $(X,{\mathcal E})$
is called an augmented tree.
\end{Def}

Following from {\cite {LaWa09}},  we say that a  path $\pi(x,y)$ is
a {\it horizontal geodesic} if it is a geodesic and  consisting of
horizontal edges only. A path is called a {\it canonical geodesic}  if
there exist $u,v\in \pi(x,y)$ such that:

\medskip

(i)\  $\pi(x,y)=\pi(x,u)\cup\pi(u,v)\cup\pi(v,y)$ with $\pi(u,v)$ a
horizontal path and $\pi(x,u),\\ \pi(v,y)$ vertical paths;

(ii)\ for any geodesic path $\pi^{\prime}(x,y)$,
$\text{dist}(o,\pi(x,y))\leq \text{dist}(o,\pi^{\prime}(x,y))$.

\medskip

Note that condition $(ii)$ is to require the horizontal part of the
canonical geodesic to be on the highest level.  The following
theorem is due to Lau and Wang \cite{LaWa09}.

\medskip

\begin{theorem}{\label{th2.3}} Let $X$ be an augmented tree. Then

\medskip

(i)  Let $\pi(x,y)$ be a canonical geodesic, then $ |x\wedge y| =
l-h/2$, where $l$ and $h$ are respectively the level and length of
the horizontal part of the geodesic.

\medskip

(ii)  $X$ is hyperbolic if and only if the lengths of horizontal
geodesics are uniformly bounded.
\end{theorem}

\medskip

 The major application of the augmented trees is  to identify their boundaries with the self-similar sets.
 We assume a self-similar set $K$ is generated by an iterated function system (IFS) $\{S_j\}_{j=1}^m$ on ${\Bbb
 R}^d$.  Let $X=\bigcup_{n=0}^{\infty}
\{1, \dots, m\}^n$ be the symbolic space representing the IFS (by
convention, $\{1, \dots, m\}^0 = \emptyset$, and we still denote it
by $o$ ). For ${\bf u} = i_1\cdots i_k$, we denote by $S_{\bf
u}=S_{i_1} \circ \cdots \circ S_{i_k}$. Let ${\mathcal E}_v$ be  the
set of vertical edges corresponding to the nature tree structure on
$X$ with $\emptyset$ as a root. We define the horizontal edges on
$X$ by
\begin{equation*}
{\mathcal E}_h =\left\{({\mathbf u},{\mathbf v}): |{\mathbf u}| =
|{\mathbf v}|,  {\mathbf u} \ne {\mathbf v}\ \text{and} \
K_{\mathbf u}\cap K_{\mathbf v} \ne \emptyset \right\},
\end{equation*}
where $K_{\mathbf u} = S_{\mathbf u}(K)$. Let ${\mathcal
E}={\mathcal E}_v \cup {\mathcal E}_h$, then $(X,{\mathcal E})$ is
an augmented tree induced by the self-similar set.

\medskip
It was proved in \cite{LaWa09} that under the {\em open set
condition (OSC)}, the augmented tree is hyperbolic and its
hyperbolic boundary can be identified with the self-similar set.
More generally, Wang  \cite {Wa12} extended these results to the
{\em weak separation condition (WSC)}, the definition was introduced
by Lau and Ngai \cite{LaNg99} to study the  multifractal structure
for an IFS with overlaps.

\medskip

In \cite{LaLu12}, we have discussed that the choice of the
horizontal edges for the augmented tree is quite flexible.
Sometimes, we prefer to use another setting by replacing $K$ with a
bounded closed invariant set $J$ (i.e. $S_i(J)\subset J$ for each
$i$), namely
\begin{equation*}
{\mathcal E}_h =\left\{({\mathbf u},{\mathbf v}): |{\mathbf u}| =
|{\mathbf v}|,  {\mathbf u} \ne {\mathbf v}\ \text{and} \ J_{\mathbf
u}\cap J_{\mathbf v} \ne \emptyset \right\},
\end{equation*}  where $J_{\bf u}:= S_{\bf u}(J)$.
We can take $J=K$ as before or in many situations, take
$J=\overline{U}$ for the $U$ in the OSC. The hyperbolicity and
identity still hold by adopting the same proof.

\medskip

We remark that the augmented tree $(X, \mathcal E)$ depends on the
choice of the bounded invariant set $J$.  But under the OSC or WSC,
the hyperbolic boundary is the same as they can be identified with
the underlying self-similar set. In the following section, we will
discuss the augmented tree and its hyperbolic structure induced by a
Moran set.

\end{section}

\bigskip

\begin{section}{\bf Moran sets as hyperbolic boundaries}

Recall that ${\mathcal M}:= {\mathcal M}(J, \{n_k\}, \{r_k\})$
denotes the class of the Moran sets satisfying (i)-(iv) of
Definition \ref{def1.1} and $r :=\inf_k r_k>0$.

\begin{Lem}\label{OSClem}
Suppose $E\in {\mathcal M}$ is a Moran set. Then for any $b>0$,
there exists a constant $c>0$ (depending on $b$) such that for any
$n$ and a set $V\subset {\mathbb R}^d$ with diameter $|V|\leq br^n$,
we have
$$\#\{{\mathbf i}\in X_n: J_{\mathbf i}\cap V\ne \emptyset\}\leq c.$$
\end{Lem}

\begin{proof}
The proof is the same as the case for a self-similar set satisfying
OSC by applying a geometrical lemma (see Lemma 9.2 of Falconer
\cite{Fa90}).
\end{proof}

We now construct a graph on $X=\bigcup_{n\geq 0} X_n$.  For each
word ${\mathbf u}\in X_n$, there exists a unique word ${\mathbf
v}\in X_{n-1}$ and a finite word ${\mathbf v'}$ such that ${\mathbf
u}=\mathbf{vv'}$. We denote the unique ${\mathbf v}$ by ${\mathbf
u}^{-1}$ and call it the first ancestor of ${\mathbf u}$,
inductively ${\mathbf u}^{-k}= \big({\mathbf u}^{-(k-1)}\big)^{-1}$.
It follows that ${\mathbf u}, {\mathbf u}^{-1},\dots, {\mathbf
u}^{-k}$ is a path from ${\mathbf u}^{-k}$ to ${\mathbf u}$. The
natural tree structure on $X$ is to take $\emptyset$ as the root $o$
and to define the set of vertical edges in this way
$${\mathcal E}_v=\big\{({\mathbf u}^{-1}, {\mathbf u}): {\mathbf
u}\in X\setminus \{\emptyset\}\big\}.$$  The horizontal edge set is
defined by
\begin{equation*}
{\mathcal E}_h =\left\{({\mathbf u},{\mathbf v}): \exists \ n>0 \
\text{such that} \  {\mathbf u}\ne {\mathbf v}\in X_n \ \text{and} \
J_{\mathbf u}\cap J_{\mathbf v} \ne \emptyset \right\}.
\end{equation*}
Let ${\mathcal E}={\mathcal E}_v\cup {\mathcal E}_h$.  With the
analogous argument in the proof of Theorem 3.2 of \cite{LaWa09}, it
concludes that

\begin{Prop}
Suppose $E\in {\mathcal M}$ is a Moran set. Then the induced
augmented tree $(X,{\mathcal E})$ is a hyperbolic graph.
\end{Prop}

\begin{proof}
First we show the graph $X$ is locally finite. By Lemma
\ref{OSClem}, there exists $c>0$ such that for any $n>0$ and ${\bf
v}\in X_n$,
$$\#\{{\bf u}\in X_n: J_{\bf u}\cap J_{\bf v}\ne \emptyset\}\leq
c.$$ Therefore ${\bf v}$ has at most $c$ neighbors in the same
level, also it has one ancestor. On the other hand, let ${\bf u}$ be
a descendant of ${\bf v}$, i.e., ${\bf u}^{-1}={\bf v}$, then
$J_{\bf u}\subset J_{\bf v}$, hence $J_{\bf u}\cap J_{\bf
v}\ne\emptyset$. Let $V=J_{\bf v}$, by using Lemma \ref{OSClem}
again, we obtain
$$\#\{{\bf u}\in X: {\bf u}^{-1}={\bf v}\}\leq c'.$$
Hence for ${\bf v}\in X$, $$\deg({\bf v})=\#\{({\bf u},{\bf v})\in
{\mathcal E}: {\bf u}\in X\}\leq c+c'+1,$$ and $X$ is locally
finite.

Next we show that the lengths of the horizontal geodesics are
uniformly bounded, then Theorem \ref{th2.3} implies that
$(X,{\mathcal E})$ is  hyperbolic. Suppose otherwise, for any
integer $k>0$, there exists a horizontal geodesic $\pi({\bf
u}_0,{\bf u}_{3k})=[{\bf u}_0,{\bf u}_1,\dots,{\bf u}_{3k}]$ with
${\bf u}_i\in X_n$. We consider the $k$-th ancestors ${\bf v}_i={\bf
u}_i^{-k}$ and the path $[{\bf v}_0,{\bf v}_1,\dots,{\bf
v}_{3k}]=[{\bf u}_0^{-k},{\bf u}_1^{-k},\dots,{\bf u}_{3k}^{-k}]$.
Let $$p({\bf v}_0,{\bf v}_{3k})=[{\bf v}_{i_0},{\bf
v}_{i_1},\dots,{\bf v}_{i_{\ell}}],\quad {\bf v}_{i_j}\in \{{\bf
v}_0,\dots, {\bf v}_{3k}\},$$ be the shortest horizontal path
connecting ${\bf v}_0$ and ${\bf v}_{3k}$. By the geodesic property
of $\pi({\bf u}_0,{\bf u}_{3k})$, it is clear that $$\ell=|p({\bf
v}_0,{\bf v}_{3k})|\geq |\pi({\bf u}_0,{\bf u}_{3k})|-2k=k.$$ Now
choose $k\geq c$ such that $(3k+1)r^k\leq 1$, where $c$ is as in
Lemma \ref{OSClem}. Let
$$V={\bigcup}_{i=0}^{3k}J_{{\bf u}_i}.$$
From $|J_{{\bf u}_i}|\leq r^n |J|, i=0,1,\dots,3k$, it is
straightforward to show that
$$|V|\leq (3k+1)r^n|J|\leq r^{n-k}|J|.$$
Note that $J_{{\bf u}_i}\subset J_{{\bf v}_i}$, we see that $J_{{\bf
v}_{i_j}}\cap V \ne \emptyset$ for each $j=0,1,\dots, \ell$. It
follows that $$\#\{{\bf v}\in X_{n-k}: J_{\bf v}\cap V\ne
\emptyset\}\geq \ell+1> k\geq c.$$ That contradicts Lemma
\ref{OSClem} and the proof is complete.
\end{proof}

\medskip

Remember that the hyperbolic boundary $\partial X$ is a compact set
under the metric $\rho_a$ and any element $\xi\in \partial X$ is
often  identified as an equivalence class of geodesic rays
$\pi[{\mathbf u}_1,{\mathbf u}_2,\ldots]$ in $X$. Now we give the
main theorem of this section which extends the results on
self-similar sets of Lau and Wang \cite{LaWa09} to that on Moran
sets.  First  a condition (H) is needed.

\medskip

\noindent {\em Condition (H):}\quad  there is a positive constant
$C'$ such that for any $n\geq 1$ and $\mathbf {u, v}\in X_n$, either
$J_{\mathbf u}\cap J_{\mathbf v}\ne \emptyset$ or
$\text{dist}(J_{\mathbf u}, J_{\mathbf v})\geq C'r^n$.

\medskip

There are many standard self-similar sets satisfying condition (H),
for example, the generating IFS has the OSC and all the parameters
of the similitudes are integers. However there are also examples
that condition (H) is not satisfied (see \cite{LaWa09},
\cite{Wa12}).

\begin{theorem} \label{thm. of moran sets}
Let $E\in {\mathcal M}$ be a Moran set, $(X,{\mathcal E})$ be the
induced augmented tree. Then  there exists  a bijection $\Phi:
\partial X \to E $ satisfying:
$$
|\Phi(\xi)-\Phi(\eta)|\leq C\rho_a(\xi,\eta)^{\alpha},\quad\text{for
any}~ \xi,\eta\in \partial X,
$$
where $\alpha=-\log r/a$. In this case $\partial X$ is homeomorphic
to $E$.

If the additional condition (H) holds,  then this $\Phi$ has the
following H\"older equivalent property:
\begin{equation}\label{eq.holder homeo.}
C^{-1}|\Phi(\xi)-\Phi(\eta)|\leq \rho_a(\xi,\eta)^{\alpha}\leq
C|\Phi(\xi)-\Phi(\eta)|.
\end{equation}
\end{theorem}

\begin{proof}
Let ${\mathbf u}_0=\emptyset$, for any geodesic ray $\xi=
\pi[{\mathbf u}_1,{\mathbf u}_2,\ldots]$, define
$$\Phi(\xi)=\lim_{n\rightarrow \infty} S_{{\mathbf u}_n}(x_0)$$
for some $x_0\in J$. It is easy to show the mapping is well-defined.
Indeed, if two geodesic rays $\xi= \pi[{\mathbf u}_0,{\mathbf
u}_1,\ldots], \eta= \pi[{\mathbf v}_0,{\mathbf v}_1,\ldots]$ are
equivalent, then there exists a constant $c>0$ such that
$$d({\mathbf u}_n,{\mathbf v}_n)\leq c\delta$$ for all $n\geq 0$, where $\delta>0$
depends only on the hyperbolicity of the graph $X$ \cite{Wo00}. Let
${\mathbf u}_n={\mathbf t}_0,{\mathbf t}_1,\dots,{\mathbf
t}_k={\mathbf v}_n$ be a canonical geodesic from ${\mathbf u}_n$ to
${\mathbf v}_n$, then $k\leq c\delta$. The canonical geodesic can be
written in three parts, two vertical and one horizontal parts:
${\mathbf t}_0,\dots,{\mathbf t}_i$; ${\mathbf t}_i,\dots,{\mathbf
t}_j$ and ${\mathbf t}_j,\dots,{\mathbf t}_k$. For the horizontal
part, we assume that ${\mathbf t}_i,\dots,{\mathbf t}_j\in
X_{\ell_n}$, i.e. the $\ell_n$-th level. Note that $$J_{{\mathbf
t}_0}\subset J_{{\mathbf t}_1}\subset \cdots \subset J_{{\mathbf
t}_i}\quad\text{and}\quad J_{{\mathbf t}_j}\supset J_{{\mathbf
t}_{j-1}}\supset \cdots \supset J_{{\mathbf t}_k}.$$ Then we have
$$|S_{{\mathbf u}_n}(x_0)-S_{{\mathbf t}_i}(x_0)|\leq |J_{{\mathbf
t}_i}|\leq r^{\ell_n}|J|$$ and $$|S_{{\mathbf v}_n}(x_0)-S_{{\mathbf
t}_j}(x_0)|\leq |J_{{\mathbf t}_j}|\leq r^{\ell_n}|J|.$$ Since the
horizontal part $[{\mathbf t}_i,\dots,{\mathbf t}_j]$ lies in the
same horizontal level $\ell_n$, it follows that $r_{{\mathbf
t}_i},\dots,r_{{\mathbf t}_j}\leq r^{\ell_n}$ and $j-i\leq c\delta$,
hence $$|S_{{\mathbf t}_i}(x_0)-S_{{\mathbf t}_j}(x_0)|\leq
(j-i+1)r^{\ell_n}|J|\leq (c\delta+1) r^{\ell_n}|J|.$$ Combining the
above estimates together, we conclude $$|S_{{\mathbf
u}_n}(x_0)-S_{{\mathbf v}_n}(x_0)|\leq Cr^{\ell_n}$$ for some
constant $C>0$.  Obverse that $\lim_{n\rightarrow
\infty}\ell_n=\lim_{n\rightarrow \infty} |{\mathbf
u}_n\wedge{\mathbf v}_n|=+\infty.$  Consequently,
$\lim_{n\rightarrow \infty} S_{{\mathbf
u}_n}(x_0)=\lim_{n\rightarrow \infty} S_{{\mathbf v}_n}(x_0)$ and
$\Phi$ is well-defined.

For any $x\in E$, there exists an infinite word $i_1i_2\cdots\in
D_{\infty}$ such that $$\lim_{n\rightarrow \infty} S_{i_1i_2\cdots
i_n}(x_0)=x.$$ Let ${\mathbf u}_0=\emptyset$ and for every $n\geq
1$, there exists a unique $k_n$ such that $r_{1}\cdots
r_{{k_n}}\leq r^n < r_{1}\cdots r_{{k_n}-1}$, denote ${\mathbf
u}_n=i_1\cdots i_{k_n}$ and $\xi=\pi[{\mathbf u}_0,{\mathbf
u}_1,\ldots]$. Then $\xi\in\partial X$ and $\Phi(\xi)=x$ which
proved that $\Phi$ is surjective. If $\Phi(\xi)=\Phi(\eta)=x\in E$,
then $x\in J_{{\mathbf u}_n}\cap J_{{\mathbf v}_n}$ for $n\geq 0$,
hence $d({\mathbf u}_n,{\mathbf v}_n)\leq 1$, and $\xi, \eta$ are
equivalent. Whence $\Phi$ is bijective.

In the following, we show that $\Phi$ is H\"older continuous. Then
being a bijective continuous map, $\Phi$ is a homeomorphism.  Let
$\xi= \pi[{\mathbf u}_0,{\mathbf u}_1,{\mathbf u}_2,\ldots], \eta=
\pi[{\mathbf v}_0,{\mathbf v}_1,{\mathbf v}_2,\ldots]$ be any two
non-equivalent geodesic rays in $X$. Then there is a bilateral
geodesic $\gamma$ joining $\xi$ and $\eta$ \cite{Wo00}. We assume
that it is canonical:
\begin{equation}\label{equation3.2}
\gamma=\pi[\dots,{\mathbf u}_{n+1},{\mathbf u}_n,{\mathbf
t}_1,\dots,{\mathbf t}_{\ell},{\mathbf v}_n,{\mathbf v}_{n+1},\dots]
\end{equation}
with ${\mathbf u}_n,{\mathbf t}_1,\dots,{\mathbf t}_{\ell},{\mathbf
v}_n\in X_n$, i.e. $n$-th level. It follows that $$|S_{{\mathbf
u}_n}(x_0)-S_{{\mathbf v}_n}(x_0)|\leq (\ell+2) r^n|J|.$$ By the
hyperbolicity of the augmented tree $X$, $\ell$ is uniformly bounded
by a constant which depends only on the graph. Since $\Phi(\xi)\in
J_{{\mathbf u}_k}$ and $\Phi(\eta)\in J_{{\mathbf v}_k}$ for all
$k\geq 0$, we get
$$|\Phi(\xi)-S_{{\mathbf u}_n}(x_0)|,\quad |\Phi(\eta)-S_{{\mathbf v}_n}(x_0)|
\leq r^n|J|.$$ Thus $$|\Phi(\xi)-\Phi(\eta)|\leq
|\Phi(\xi)-S_{{\mathbf u}_n}(x_0)|+|S_{{\mathbf
u}_n}(x_0)-S_{{\mathbf v}_n}(x_0)|+|\Phi(\eta)-S_{{\mathbf
v}_n}(x_0)|\leq C_1 r^n.$$ Since it is a bilateral canonical
geodesic,
 we have $|\xi\wedge\eta|=n-(\ell+1)/2$ and $\ell$ is uniformly bounded. By
 using $\rho_a(\xi,\eta)=\exp(-a |\xi\wedge\eta|)$,
it yields that
$$|\Phi(\xi)-\Phi(\eta)|\leq C\rho_a(\xi,\eta)^{\alpha}.$$

For the second part, if the additional condition holds,  assume that
$\xi \ne \eta$. Since $\gamma$ in (\ref{equation3.2}) is a geodesic,
it follows that $({\mathbf u}_{n+1},{\mathbf v}_{n+1})\notin
{\mathcal {E}}_h$, and hence $J_{{\mathbf u}_{n+1}}\cap J_{{\mathbf
v}_{n+1}}=\emptyset$ which implies
$$|\Phi(\xi)-\Phi(\eta)|\geq dist(J_{{\mathbf u}_{n+1}},J_{{\mathbf v}_{n+1}})\geq
C'r^{n+1},$$ and the theorem follows in view of the definition of
the metric $\rho_a$.

\end{proof}

\end{section}

\bigskip

\begin{section}{\bf Lipschitz equivalence}

Two compact metric spaces $(X,d_X)$ and $(Y,d_Y)$ are said to be
{\em Lipschitz equivalent}, and denote by $X\simeq Y$, if there is a
bi-Lipschitz map  $\sigma$ from $X$ onto $Y$, i.e., $\sigma$ is a
bijection and  there is a constant $C>0$ such that
$$
C^{-1}d_X(x,y)\leq d_Y(\sigma(x),\sigma(y))\leq Cd_X(x,y)\quad
\quad \forall  \ x,y\in X.
$$ Following \cite{LaLu12}, we have

\begin{Def}
Let $X$ and $Y$ be two hyperbolic graphs and let $\sigma: X
\rightarrow Y$  be a bijective map. We say that $\sigma$ is a
near-isometry if there exists $c>0$ such that
$$
 \big ||\pi(\sigma(x),\sigma(y))|-|\pi(x,y)|\big| \leq c \qquad \forall  \ x, y \in X.
$$
\end{Def}

\begin{Prop}(\cite{LaLu12}) \label{th4.2}
 Let $X$, $Y$ be two hyperbolic augmented trees that are equipped with the
hyperbolic metrics with the same parameter  $a$ (as in
(\ref{eq2.1})). Suppose there exists a near-isometry  $\sigma : X
\to Y$, then \ $\partial X \simeq \partial Y$.
\end{Prop}

According to conditions (iii) and (iv) of Definition \ref{def1.1}, it follows that every basic set $J_{\mathbf i}$ (say ${\mathbf i}\in D_{k-1}$)  gives rise to the same number of subsets (say $n_k$) under the same contraction ratio (say $r_k$)  from one step to the next. That, together with (\ref{add.equ}), implies that, for an augmented tree $X= \bigcup_{n\geq 0}X_n$,   any two words of the level $X_n$  generate the same number of offsprings in the level $X_{n+1}$.

\medskip

By a horizontal connected components of an augmented tree $X$, we
mean a maximal connected horizontal subgraph on some level $X_n$.
Let $\mathscr{C}$ be the set of all horizontal connected components
of $X$.   For $T\in \mathscr {C}$, say,  it lies in the level $X_n$, we
let
$$T\Sigma_n = \{{\bf ui} \in X_{n+1}: {\bf u} \in T, {\bf i} \in
\Sigma_n\}$$  denote the set of offsprings of $T$ in the level $X_{n+1}$, where $\Sigma_n$ denotes the suffix
set of words. By the previous argument, we note that the $\Sigma_n$ only depends on the level $X_n$,  i.e.,  $\Sigma_n=\Pi_{j=a(n)}^{b(n)}\{1,2,\dots, n_j\}$ where $a(n)\leq b(n)$ are positive integers depending on $n$ only.  If no confusion occurs, we write $\Sigma:= \Sigma_n$ for simplicity.

\medskip

With the above notation, we introduce a key concept of this section.

\begin{Def} \label{def4.3}
An augmented tree $X$ is called rearrangeable if $\max\{\#T: T\in
{\mathscr C}\}< \infty$, and for any $T\in {\mathscr C}$ with
$\#T=b$,  its offsprings $T\Sigma$ of the next level can be decomposed into $b$ groups as following
$$T\Sigma =\left({\bigcup}_{k\in \Lambda_1}Z_{k}\right)\ \cup \ \cdots \cup \
\left({\bigcup}_{k\in \Lambda_{b}}Z_{k}\right)$$ such that every
$Z_k\in {\mathscr C}$ consists of the offsprings of $T$ and the
total size of every group is equal to $\#\Sigma$.
\end{Def}

In fact,  the concept of `rearrangeable  augment tree' coincides
with the two additional conditions (v) and (vi) of Moran sets
defined in the introduction. Recall that ${\mathcal M}':={\mathcal
M}(J, \{n_k\}, r)$ is the collection  of all Moran sets satisfying
conditions (v) and (vi). It is easy to verify that if a Moran set
$E\in {\mathcal M}'$  then the induced augmented tree $X$ is
rearrangeable. Moreover, if we assume $n_k \equiv n$ and $r_k \equiv r$ and choose
a fixed IFS in Definition \ref{def1.1}, then the Moran set
degenerates to the self-similar set with equal ratio and the
discussion of this section goes back to the case considered by
\cite{LaLu12}.

\begin{theorem}\label{th4.4}
Suppose the augmented tree $(X, {\mathcal E})$ is rearrangeable.
Then there is a near-isometry between $(X, {\mathcal E})$ and $(X,
{\mathcal E}_v)$ so that $(X, {\mathcal E})\simeq (X, {\mathcal
E}_v)$.
\end{theorem}

\begin{proof}
Let $X=(X,{\mathcal E}),~ Y=(X,{\mathcal E}_v)$. It suffices to
construct a near-isometry $\sigma$ between $X$ and $Y$, and hence
$\partial(X, {\mathcal E} )\simeq
\partial(X, {\mathcal E}_v )$ by Proposition \ref{th4.2}.  We define this $\sigma$ to be a
one-to-one mapping  from $X_n$ (in $X$) to $X_n$ (in $Y$)
inductively as follows: Let
$$
\sigma(o)=o, \quad \hbox {and} \quad \sigma(x)=x, \ x\in X_1.
$$
Suppose  $\sigma$ is defined on the level $n$ (i.e., $X_n$) such
that for every horizontal connected component $T$, $\sigma (T)$ has
the same parent, i.e.,
\begin{equation} \label {eq4.2}
\sigma (x)^{-1} = \sigma (y)^{-1} \qquad  \forall \ x, y \in T
\subset X_n
\end{equation}
(see Figure 1). To define the map $\sigma$ on  $X_{n+1}$, we note
that $T$ in $X_n$ gives rise to horizontal connected components in
$X_{n+1}$. We can write
$$
T\Sigma= {\bigcup}_{k=1}^\ell Z_k.
$$
where $\Sigma$ is the set of suffixes such that $T\Sigma \subset
X_{n+1}$, and $Z_k$ are horizontal connected components consisting
of offsprings of $T$.  Let $\#T=b$ and $\#\Sigma=m$, by the
rearrangeable condition, ${\bigcup}_{k=1}^{\ell}Z_k$ can be
rearranged as $b$ groups so that total size of every group is equal
to $m$, namely,
\begin{equation} \label {eq4.3}
{\bigcup}_{k=1}^{\ell}Z_k={\bigcup}_{k\in \Lambda_1}Z_{k}\ \cup \
\cdots \cup \ {\bigcup}_{k\in \Lambda_{b}}Z_{k}.
\end{equation}
Note that each set on the right has $m$ elements.
\begin{figure}[h]
  \centering
  \includegraphics[width=12cm]{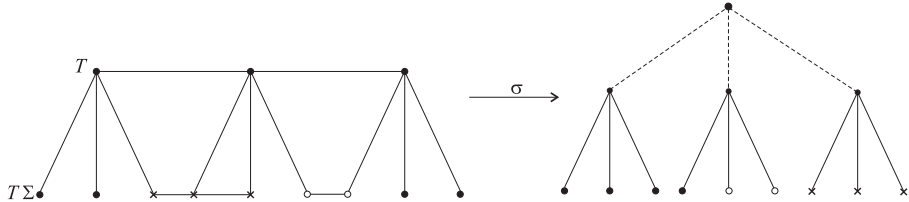}
\caption{ An illustration of the rearrangeable condition by
$\sigma$.}\label{fig.re}
\end{figure}

For the  connected component  $T = \{{\mathbf i}_1,  \dots ,
{\mathbf i}_{b}\} \subset X_n$, we have defined $\sigma$ on $X_n$
and $ \sigma (T) = \{{\mathbf j}_1 = \sigma ({\mathbf i}_1), \ \dots
,\ {\mathbf j}_{b} = \sigma ({\mathbf i}_{b})\}. $ In view of (\ref
{eq4.3}), we define $\sigma$ on $T\Sigma =
{\bigcup}_{k=1}^{\ell}Z_k$ by assigning each ${\bigcup}_{k\in
\Lambda_s}Z_{k}$ (it has $m$ elements) the $m$ descendants of
${\mathbf j}_s$ (see Figure \ref{fig.re}). It is clear that $\sigma$
is well-defined on $T\Sigma$ and satisfies (\ref{eq4.2}) for $x,y
\in T\Sigma$. We apply the same construction of $\sigma$ on the
offsprings of every horizontal connected component in $X_n$. It
follows that $\sigma$ is well-defined  and satisfies (\ref{eq4.2})
on $X_{n+1}$. Inductively, $\sigma$ can be defined from $X$ to $Y$
and  is bijective.

\medskip

Finally we show that  $\sigma$ is indeed a near-isometry and
complete the proof.  Since $\sigma: X \to Y$ preserves the levels,
hence without loss of generality, it suffices to prove the
near-isometry for $\mathbf{x,y}$ belong to the same level. Let
$\pi(\mathbf{x,y})$  be the canonical geodesic connecting them,
which can be written as
$$
\pi(\mathbf {x,y})=[{\mathbf x}, {\mathbf u}_1,\dots, {\mathbf u}_n,
{\mathbf t}_1,\dots, {\mathbf t}_k, {\mathbf v}_n,\dots, {\mathbf
v}_1, {\mathbf y}]
$$
where $[{\mathbf t}_1,\dots, {\mathbf t}_k]$ is the horizontal part
and $[{\mathbf x}, {\mathbf u}_1,\dots, {\mathbf u}_n, {\mathbf
t}_1],~ [{\mathbf t}_k, {\mathbf v}_n,\dots, {\mathbf v}_1, {\mathbf
y}]$ are vertical parts. Clearly, $\{{\mathbf t}_1,\dots, {\mathbf
t}_k\}$ must be included in one horizontal connected component of
$X$, we denote it by $T'$. With the notation as in Theorem
\ref{th2.3}(i), it follows  that for ${\mathbf x}\ne {\mathbf y}\in
X$,
$$
|\pi(\mathbf {x,y})| =|{\mathbf x}|+|{\mathbf y}|-2l+h,  \qquad |\pi
(\sigma({\mathbf x}),\sigma({\mathbf y}))|=|\sigma({\mathbf
x})|+|\sigma({\mathbf y})|-2l^{\prime}+h^{\prime}.
$$
We have $$\big ||\pi (\sigma({\mathbf x}),\sigma({\mathbf
y}))|-|\pi(\mathbf {x,y})|\big |\leq |h-h'|+2|l'-l|\leq k+2|l'-l|$$
where $k$ is a hyperbolic constant as in Theorem \ref{th2.3}(ii). If
$T'$ is a singleton, then
$$|l'-l|=0.$$
If  $T'$ contains more than one point, then the elements of
$\sigma(T')$ share the same parent. Then the confluence of
$\sigma({\mathbf x})$ and $\sigma({\mathbf y})$ (as a tree) is
$\sigma ({\mathbf x)}^{-1} \ ( = \sigma ({\mathbf y)}^{-1})$. Hence
$$|l'-l|=1.$$
Consequently,
$$
\big ||\pi(\sigma({\mathbf x}),\sigma({\mathbf y}))|-|\pi(\mathbf
{x, y})|\big |\leq k+2.
$$
This completes the proof that $\sigma$ is a near-isometry and the
theorem is established.

\end{proof}

\begin{Cor}
Under the assumption on the above theorem,  then  $(\partial (X,
{\mathcal E}), \rho_a)$  is totally disconnected.
\end{Cor}

By Theorem \ref{th4.4},  we obtain the following Lipschitz
equivalence on  Moran sets. Furthermore, Theorem \ref{th1.3} can be proved
as well.

\medskip

\begin{theorem}\label{th4.6}
Let $E,E'\in{\mathcal M}$ be two moran sets, and satisfy condition
(H). Assume the associated augmented trees both are rearrangeable.
Then $E$ and $E'$ are Lipschitz equivalent.
\end{theorem}

\begin{proof}
It follows from Theorem \ref{th4.4} that
\begin{equation*}
\partial(X, {\mathcal E}) \simeq \partial(X, {\mathcal E}_v) = \partial( Y,
{\mathcal E}_v )\simeq \partial(Y, {\mathcal E})
\end{equation*}
(for the respective metrics $\rho_a$). Let $\varphi:
\partial(X, {\mathcal E})  \to \partial(Y, {\mathcal E})$
be the bi-Lipschitz map. With no confusion, we just denote these two
boundaries by $\partial X$, $\partial Y$ as before.

By Theorem \ref{thm. of moran sets}, there exist two bijections
$\Phi_1:\partial X\to E$ and $\Phi_2:\partial Y\to E'$ satisfying
(\ref{eq.holder homeo.}) with constants $C_1,C_2$, respectively.
Define $\tau: E \to E'$ as
$$
\tau = \Phi_2\circ \varphi\circ \Phi_1^{-1}.
$$
Then
$$
\begin {aligned}
|\tau(x) -\tau (y)| & \leq
C_2\ \rho_a(\varphi\circ\Phi_1^{-1}(x),\varphi\circ\Phi_1^{-1}(y))^\alpha\\
&\leq C_2C_0^\alpha\ \rho_a(\Phi_1^{-1}(x),\Phi_1^{-1}(y))^\alpha\\
&\leq C_2C_0^\alpha C_1\ |x-y|.
\end{aligned}
$$
Let $C' = C_2C_0^\alpha C_1$, then
$$
|\tau (x) - \tau (y)| \leq C'|x -y|.
$$
Similarly, we have ${C'}^{-1} |x-y| \leq |\tau (x) - \tau (y)|$.
Therefore $\tau : E \to E'$ is a bi-Lipschitz map.
\end{proof}

\medskip

Under the H\"older equivalent property (\ref{eq.holder homeo.}), the
proof of the above theorem still yields an interesting result.

\begin{Cor}
Let $E,E'\in{\mathcal M}$ be two Moran sets satisfying condition
(H), let $\partial X $ and $\partial Y $ be their hyperbolic
boundaries, respectively. Then
\begin{equation*}
E \simeq E'\ \Leftrightarrow \ \partial X  \simeq \partial Y.
\end{equation*}
\end{Cor}

\end{section}

\bigskip
\bigskip

\noindent {\it Acknowledgements}:  The author would like to thank
Professor Ka-Sing Lau  for many valuable comments.

\end{document}